\documentclass{ws-procs9x6}
\usepackage{graphicx}
\usepackage{graphicx}
\usepackage{pstricks}

\newcommand{\bq}{\mathbf q}

\newcommand{\cH}{\mathcal H}

\newcommand{\g}{{\mathfrak{g}}}

\newcommand{\n}{{\mathfrak{n}}}

\newcommand{\cC}{\mathcal C}

\newcommand{\sal}{^{(\al)}}
\newcommand{\sbe}{^{(\al)}}
\newcommand{\ev}{{\rm ev}}
\renewcommand{\sl}{{\mathfrak{sl}}}
\newcommand{\KR}{{{\rm KR}}}

\newcommand{\C}{{\mathbb C}}
\newcommand{\Z}{{\mathbb Z}}

\newcommand{\N}{{\mathbb N}}
\newcommand{\Hom}{{\rm Hom}}

\newcommand{\bt}{{\mathbf t}}

\renewcommand{\bm}{{\mathbf m}}
\newcommand{\bn}{{\mathbf n}}
\newcommand{\bx}{{\mathbf x}}

\newcommand{\bu}{{\mathbf u}}

\newcommand{\blambda}{{\boldsymbol \lambda}}
\newcommand{\bmu}{{\boldsymbol \mu}}
\newcommand{\al}{{\alpha}}

\begin{document}

\title{A pentagon of identities, graded tensor products, and the
  Kirillov-Reshetikhin conjecture}

\author{Rinat Kedem}

\address{Department of Mathematics, University of Illinois\\
  Urbana, IL 61821, USA\\
$^*$E-mail: rinat@illinois.edu}

\begin{abstract}
  This paper provides a brief review of the relations between the
Feigin-Loktev conjecture on the dimension of graded tensor products of $\g[t]$-modules, the Kirillov-Reshetikhin conjecture, the combinatorial ``$M=N$" conjecture, their
proofs for all simple Lie algebras, and a pentagon of identities which
results from the proof.
\end{abstract}

\vskip.25in
\centerline{Dedicated to T. Miwa on the occasion of his 60th birthday}

\bodymatter

\section{Introduction}\label{intro}

This paper reviews work which followed\cite{AKS,AK,DFK} the author's fruitful
collaboration with T. Miwa and colleagues
\cite{FKLMM,FJKLM}. This work was inspired by the work of Feigin and Loktev on fusion 
products\cite{FL}. The series of results described here
finally culminated in a proof\cite{AK,DFK} of the Feigin-Loktev conjecture concerning the
graded character of the (non-level restricted) fusion product, in the case of special modules
known as Kirillov-Reshetikhin modules. The purpose of this article is to make clear the sequence of connections and relations between the various results which lead to the proof.

The fusion product character first appeared in the `80's, in the work on the completeness conjecture for Bethe Ansatz states\cite{KR} for the generalized Heisenberg models. 
The completeness conjecture is one version of what later became known as the Kirillov-Reshetikhin conjecture, and involves the first appearance of an object called the
 generalized Kostka polynomial. The (generalized) Kostka number gives the decomposition coefficients of tensor products of KR-modules into irreducible components. Although much work was later published on the subject, the conjecture in its original, combinatorial form --
counting solutions of the Bethe equations -- was only proved to be true
in special cases. In other cases, a similar but not manifestly
positive formula\cite{HKOTY,Nakajima,Hernandez} was shown to hold.

The key to proving both the Kirillov-Reshetikhin conjecture and the Feigin-Loktev conjecture is a
combinatorial identity, the equality of two polynomials in $q$, one
written as an alternating sum, and the other as a sum of positive
terms. The deeper meaning of this identity remains mysterious, but its
proof\cite{DFK} using purely combinatorial means finally implies
several equalities, proving the conjectures above for any simple Lie
algebras. 


\subsection{The objects of interest}
We describe several objects and relations between their dimensions.
(Section 2 contains a fuller discussion of several of these as necessary).

\subsubsection{Kirillov-Reshetikhin modules} These are finite-dimensional modules of the quantum affine algebra $U_q(\widehat{\g})$ or the Yangian $Y(\g)$. Let $\g$ be a simple Lie
algebra of rank $r$ with Cartan matrix $C$. Consider the irreducible,
finite-dimensional $\g$-module $V$ with a highest weight which is a
non-negative multiple of a fundamental weight. The Yangian $Y(\g)$
contains $\g$ as a subalgebra. The irreducible $Y(\g)$-module induced from $V$ is
called a Kirillov-Reshetikhin module \cite{KR}. It is finite-dimensional, and its $\g$-highest weight is that of $V$. In the case where $\g=A_n$, it is equal to $V$ as a $\g$-module. In other cases, the restriction to a $\g$-module may or may not be irreducible, but in that case, $V$ is always a component in the decomposition, with multiplicity 1, and with the highest weight.

Equivalently, one may consider Kirillov-Reshetikhin modules for the
quantum affine algebra $U_q(\widehat{\g})$ and their similar decomposition
into $U_q(\g)$-modules\cite{Chari}. These are also referred to as KR-modules.

We denote the KR-modules by $\KR_{\al,m}(\zeta)$, where $1\leq \al
\leq r$ and $m$ is a positive integer. These have a $\g$-highest
weight of the form $m\omega_\al$, where $\omega_\al$ is a
fundamental weight of $\g$. The parameter $\zeta$ is a complex number
which is called the spectral parameter.

\subsubsection{Chari's graded $\g[t]$-modules}
These are modules of the current algebra $\g[t]$, defined as a
quotient of $U(\n_-[t])$ by an ideal generated by relations\cite{Chari} (see Equations \eqref{chari},\eqref{highest}).  The
relations are the $q\to 1$ limits of the similar relations which hold
in the quantum case for Kirillov-Reshetikhin modules. These modules
also have a $\g$-highest weight equal to a multiple of one of the fundamental
weights of $\g$, as in the quantum algebra case. We denote this module
by $C_{\al,m}(\zeta)$, where the highest weight is again $m \omega_\al$.

\subsubsection{Decomposition of tensor products}
We observe that by general deformation arguments, the dimension of
KR-modules is bounded from above by that of Chari's modules.

More precisely, the decomposition coefficients of the KR-modules, and
therefore their tensor products at generic values of the spectral
parameters, into irreducible $\g$-modules are bounded from above by
those of Chari's modules. That is, choose a sequence of non-negative
integers $\bn=\{n_{m}\sal:\ 1\leq \al\leq r, m>0\}$ and consider the
multiplicities $M^{Y}_{\lambda,\bn}$ and $M^{\g[t]}_{\lambda,\bn}$, defined by
$$
M_{\lambda,\bn}^{Y}=\dim \Hom_{\g} \left(\otimes_{\al,m}
  \KR_{\al,m}^{\otimes n_{m}\sal} , V(\lambda)\right),
$$
where $V(\lambda)$ is the irreducible highest weight $\g$-module with highest weight $\lambda$, and
$$
M_{\lambda,\bn}^{\g[t]}=\dim \Hom_\g \left(\otimes_{\al,m} C_{\al,m}^{\otimes n_{m}\sal} , V(\lambda)\right).
$$
(Here, we omitted the dependence of the modules on the spectral
parameter: We assume that all spectral parameters are taken at generic
values with respect to each other).
Then we have the first inequality:
\begin{equation}
M_{\lambda,\bn}^{Y}\leq M_{\lambda,\bn}^{\g[t]},
\end{equation}
which simply follows by general deformation arguments: Both are defined as quotients by some ideal, and the ideal in the limit $q\to 1$ may be smaller than that for generic  values of $q$.

\subsubsection{The combinatorial KR-conjecture: The $M$-sum
  formula}  This is a conjecture that $M^{Y}_{\lambda,\bn}$ is equal to the number of Bethe vectors.
The generalized, inhomogeneous Heisenberg spin chain has a Hilbert space which is equal, by definition, to the tensor product of Yangian modules,
$$
\mathcal H_\bn = \prod_{\al,m} \KR_{\al,m}^{\otimes n_{m}\sal}.
$$
Again, the modules are taken at generic values of the spectral
parameters, that is, pairwise not separated by an integer. (Note that the model is also well-defined for any other finite-dimensional $Y(\g)$-modules, but no Bethe Ansatz solution is known generically.)

This model has a Bethe Ansatz solution. The completeness conjecture of
Kirillov and Reshetikhin\cite{KR} states that there is a Bethe vector
for each $\g$-highest weight vector in $\mathcal H_\bn$. In particular, there is 
an explicit formula for the number of Bethe vectors, and in fact, the authors wrote down a graded formula (which we now know has a direct interpretation as a grading by the linearized energy function of the model), although at the time, the meaning of the refinement was unknown. For the $\g$-highest
weight $\lambda$, the multiplicity is the number $M_{\lambda,\bn}$
obtained as the $q\to 1$ limit of the following, grading-endowed
formula:\cite{KR}
\begin{equation}\label{Msum}
  M_{\lambda,\bn}(q) = \sum_{\bm} q^{Q({\bm},\bn)}\prod_{\al,i} \left[ \begin{array}{c} m_{i}\sal +P_{i}\sal \\ {m_{i}\sal}\end{array}\right]_q
\end{equation}
where the sum is taken over the non-negative integers 
${\bm} = \{ m_{i}\sal: 1\leq \al \leq r, i\geq 1\}$ such that $\sum_i
m_{i}\sal = m\sal$, where $m\sal$ are integers fixed by the ``zero
weight condition" 
\begin{equation}\label{weight}
\sum_\beta C_{\al,\beta} m\sal = \sum_i n_{i}\sal - \ell\sal,
\end{equation}
$\ell\sal$ being the coefficient of $\omega_\al$ in the weight
$\lambda$. Note that this sum has only a finite number of non-vanishing terms.
Let us define 
$$B_{i,j}^{(\al,\beta)}={\rm sign}({C_{\al,\beta}})\min(|C_{\al,\beta}| j,|C_{\beta,\al}| i)$$
Then the vacancy numbers $P_i\sal$ are defined as
\begin{equation}\label{vacancy}
P_i\sal = \sum_{i} \min(i,j) n_j\sal - (B \bm)_i\sal.
\end{equation}
and the quadratic form $Q(\bm,\bn)$ is
\begin{equation}\label{energy}
Q(\bm,\bn) = \frac{1}{2} \bm\cdot \mathbf P.
\end{equation}

The $q$-binomial coefficient is defined as usual, and in the limit
$q\to 1$ becomes the usual binomial coefficient. In particular, {\em the
sum is taken over the restricted set of integers $\bm$ such that
$P_{j}\sal\geq 0$.}

This provides a formula for the tensor multiplicities
$M_{\lambda,\bn}^{Y}$. It was proved in several special cases using
combinatorial means\cite{KR,KKR,KSS02,Sch05,OSS03,SS06}. In general, a
similar but not equivalent formula was known to be true, as explined
below.

\subsubsection{The HKOTY $N$-sum formula}
For general Lie algebras, and for generic $\KR$-modules, the following
formula was conjectured\cite{HKOTY}:
\begin{equation}\label{Nsum}
M_{\lambda,\bn}^{Y} =\lim_{q\to 1} N_{\lambda,\bn}(q),
\end{equation}
where $N_{\lambda,\bn}(q)$ is a modified form of the
formula \eqref{Msum}, obtained by simply removing the restriction $P_{j}\sal\geq
0$. Both the usual and the $q$-binomial coefficients are defined when
$P_j\sal<0$, but they carry a sign in that case.
The authors conjectured, after extensive testing, that all
terms coming from sets $\bm$ such that $P_j\sal<0$ for some $j,\al$
cancel, so that 
\begin{equation}\label{HKOTY}
M_{\lambda,\bn}(q)=N_{\lambda,\bn}(q).
\end{equation}

The conjecture \eqref{Nsum} holds provided that the
so-called $Q$-system \cite{KR} is satisfied by the characters of
KR-modules. It was shown by Nakajima (for simply-laced algebras)
\cite{Nakajima} and Hernandez for all other Lie algebras \cite{Her}
that the $q$-characters of KR-modules satisfy the more general
$T$-system \cite{KNS}, from which the $Q$-system follows. Hence, Equation
\eqref{Nsum} had achieved the status of a Theorem.

\subsubsection{Feigin-Loktev fusion products} The Feigin-Loktev fusion
product is a graded $\g[t]$-module \cite{FL}, which is a refinement of the usual tensor product of $\g$-module (cyclic, finite-dimensional). One chooses a
finite-dimensional cyclic $\g$-module $V$, from which one induces an
action of the current algebra $\g[t]$ localized at some complex number
$\zeta$. More specifically, one defines a graded tensor product by
choosing $N$ $\g$-modules $V_i$, each with a cyclic vector $v_i$,
localized at $N$ distinct points in $\C P$. One then defines the
fusion product as the associated graded space of the filtered space, generated by the action of $U(\g[t])$ on
the tensor product of cyclic vectors, with the grading defined by
degree in $t$. This is a graded space, and the graded components are
$\g$-modules.

Feigin and Loktev conjectured that the fusion product as a graded
space is independent of the localization parameters for sufficiently
well-behaved $\g$-modules. Moreover, they conjectured a relation
between the graded coefficients of the $\g$-module $V(\lambda)$ in the
fusion product, and the generalized Kostka polynomials\cite{KirillovShi}. This
conjecture was proved for $\sl_2$ in \cite{FJKLM}, and in greater
generality in several other works.

In particular, in \cite{AK}, we proved the following inequality, using
techniques generalized from \cite{FJKLM}. Let $\mathcal F_\bn^*$ be
the fusion product of the modules $C_{\al,m}$ with multiplicity
$n_m\sal$. This is a graded space. We define the $q$-dimension to be
the Hilbert polynomial of the graded space. Then
\begin{equation}\label{fusion}
q\hbox{-}\dim\Hom_\g \left( \mathcal F_\bn^*, V(\lambda)\right) \leq M_{\lambda,\bn}(q),
\end{equation}
where $M_{\lambda,\bn}$ is the fermionic formula of Kirillov and
Reshetikhin for the number of Bethe vectors in the generalized,
inhomogeneous Heisenberg spin chain corresponding to KR-modules
$KR_{\al,m}$ with the same multiplicities.

\begin{remark}The inequality in \eqref{fusion} arises from the following sequence of maps: One may completely characterize the dual space of functions of the fusion product in terms of symmetric functions with certain zeros and poles (we do this in Section 3). Actual calculation of the Hilbert polynomial of this space requires another injective mapping into another filtered space, whose Hilbert polynomial is the polynomial $M_{\lambda,\bn}(q)$. We do not prove the surjectivity of the map, resulting in the inequality in Equation \eqref{fusion}. 
\end{remark}

Moreover, the space $\mathcal F_\bn^*$ is the associated graded space
of the tensor product of Chari modules, which are defined as a
quotient of $U(\g[t])$ by a certain ideal. Again, by a general
deformation argument, we have that
$$
\dim \Hom_\g (\otimes C_{\al,i}^{\otimes n_{i}\sal} , V(\lambda)) \leq \dim \Hom_\g 
(\mathcal F^*_\bn,V(\lambda)).
$$

Note that the sum on the right hand side of Equation \eqref{fusion} is
manifestly positive, and therefore if, in the $q\to 1$ limit, it is
equal to the tensor product multiplicity, then we have the equality of
graded spaces also, since the left-hand side has a dimension which is
greater than or equal to the tensor product multiplicity by the
deformation argument.

\subsection{A pentagon of identities}
We have a sequence of identities and inequalities:
\begin{eqnarray*}
&&\hbox{ $\left|\Hom_\g\left(\underset{\alpha,m}{\otimes} C_{\alpha,m}^{\otimes n_{m}\sal}~,~V(\lambda)\right)\right|$ \hbox{\raisebox{20pt}{
\rotatebox[origin=c]{45}{$\leq$}
\hbox{\raisebox{20pt}{$\left|\Hom_\g\left(\mathcal F^*_\bn~,~V(\lambda)\right)\right|$}}}}
\hbox{\raisebox{20pt}{\rotatebox[origin=c]{-45}{$\leq$}}} $M_{\lambda,\bn}$} \\
&&\hbox{\hskip1in\rotatebox[origin=c]{90}{$\leq$}}\hskip 2.5in
\hbox{ \rotatebox[origin=c]{90}{$=$}$\leftarrow$final step}\\
&&\hbox{\hskip1cm $\left|\Hom_{U_q(\g)}\left(\underset{\alpha,m}{\otimes} KR_{\alpha,m}^{\otimes n_{m}\sal}~,~V(\lambda)\right)\right|$ } \hskip 1.cm \hbox{$=$} \hskip 1cm \hbox{$N_{\lambda, \bn}$}
\end{eqnarray*}


The ``final step" remaining in this pentagon was to prove the conjectured
identity \eqref{HKOTY}. The proof turns all the inequalities in the pentagon to equalities. This conjecture was proven by combinatorial
means\cite{DFK} for all simple Lie algebras.  Therefore, this provides a proof of
the completeness problem in the Bethe Ansatz known as the
Kirillov-Reshetikhin conjecture, as well as the Feigin-Loktev
conjecture for the cases of Kirillov-Reshetikhin conjecture.

\subsection{Plan of the paper}
In the following sections, we will summarize the proof\cite{AK} of the
inequality 
\eqref{fusion} and the proof of the $M=N$ conjecture\cite{DFK}. In Section
2, we give a definition of the Feigin-Loktev fusion product of Chari's
modules.  In Section 3, we summarize the proof of the inequality \eqref{fusion}
which is obtained via a functional realization of the multiplicity
space, following the ideas of B. Feigin. In Section 4, we explain the
combinatorial proof of the $M=N$ conjecture \cite{DFK}.



\section{Definitions}\label{defs}
Here, we add some details to the definitions of the representation-theoretical objects which are important in the theorems below.

Let $\g$ be a simple Lie algebra of rank $r$ and Cartan matrix
$C$. Let $\g[t]=\g\otimes \C[t]$ be the corresponding Lie algebra of
polynomials in $t$ with coefficients in $\g$.

\subsection{Finite-dimensional $\g[t]$-modules and the fusion action}
Given a complex number $\zeta$, any $\g$-module $V$ can be extended to a
$\g[t]$-module evaluation module $V(\zeta)$, with $t$ evaluated at $\zeta$.
The generators $x[n]:= x\otimes t^n$ ($x\in\g$) act on $v\in V$ as
$\pi(x[n]) v = \zeta^n x v.$

The dimension of the evaluation module is the same as that of $V$.
If $V$ is irreducible as a $\g$-module, so is $V(\zeta)$.

More generally, given a $\g[t]$-module $V$, the $\g[t]$-module
localized at $\zeta$, $V(\zeta)$, is the module on which $\g[t]$ acts
by expansion in the local parameter $t_\zeta := t-\zeta$. If $v\in
V(\zeta)$, then
$$x[n] v = x\otimes (t_\zeta + \zeta)^n v = \sum_j {n \choose j}
\zeta^j x[n-j]_\zeta v,$$ where $x[n]_\zeta := x\otimes t_\zeta$ and
$x[n]_\zeta$ acts on $v\in V(\zeta)$ in the same way that $x[n]$ acts
on $v\in V$.

Another way to write this is in terms of generating functions, for any $x\in\g$,
$$x(z) = \sum_{n\in\Z} x[n] z^{-n-1}.
$$
Then if $\zeta\in\C$, 
\begin{equation}\label{localaction}
x[n]_\zeta = \frac{1}{2\pi i} \oint_{z=\zeta} (z-\zeta)^n x(z) dz.
\end{equation}
We will also need to be able to localize modules at infinity. In that case,
\begin{equation}\label{inftyaction}
x[n]_\infty = \frac{1}{2\pi i} \oint_{z=\infty} z^{-n} x(z) dz =
\frac{-1}{2\pi i}\oint_{z=0} z^{n-2} x(z^{-1}) dz.
\end{equation}

An evaluation module $V(\zeta)$ is a special case of a localized
module, on which the positive modes $x[n]_\zeta$ with $n>0$ and
$x\in\g$ act
trivially.

Let $V$ be a cyclic $\g[t]$-module with cyclic vector $v$. Then $V$ is
endowed with a $\g$-equivariant grading inherited from the grading of
$U:=U(\g[t])$. The filtred
components of $V$ are $\mathcal F(n) = U^{\leq n}v$, where $U^{\leq
  n}$ is the subspace of homogeneous degree in $t$ bounded by $n$. The
grading on $V$ is the associated graded space of this filtration, $\underset{n\geq 0}{\oplus}\mathcal F(n)/\mathcal F(n-1)$. 

As the filtration is $\g$-equivariant, the graded components are $\g$-modules.

\subsection{Chari's KR-modules of $\g[t]$}
A special case of the construction described in the previous subsection is given as follows.
Consider $\g[t]$-modules with a highest weight $m \omega_\al$ ($m\geq
0$ and $\omega_\al$ a fundamental $\g$-weight) defined as the cyclic
module generated by a highest weight vector $v$, with relations
\begin{eqnarray}
x[n]_\zeta v &=& 0 \quad \hbox{if $x\in \n_+$ and $n\geq 0$};\nonumber
\\
h_\beta[n]_\zeta v &=& \delta_{n,0} \delta_{\al,\beta} m v;\nonumber\\ 
f_\beta[n]_\zeta v &=& 0 \quad \hbox{if $n\geq \delta_{\al,\beta}$};\label{highest}\\
f_\al[0]_\zeta^{m+1} v &=& 0.\label{chari}
\end{eqnarray}
We refer to these modules as $C_{\al,m}(\zeta)$\cite{Chari}. Their
graded version has been previously considered by Chari and Moura\cite{ChariMoura}. The graded components of the associated graded space corresponding to the filtration by homogeneous degree are, of course, $\g$-modules.

Except in the case of $\g=A_r$,
these modules are not necessarily irreducible under restriction to the action of $\g$. However,
$C_{\al,m}(\zeta)$ does have a highest weight component isomorphic to the $\g$-module $V(m
\omega_\al)$, which appears with multiplicity 1, all other components
having a smaller highest weights in the total ordering.

It has not been directly proven (except in special cases) that these
modules have the same $\g$-decomposition as the Yangian KR-modules, but this
theorem will follow from the proof of the Feigin-Loktev conjecture below.

\subsection{Fusion products and the Fegin-Loktev conjecture}
Consider a set of cyclic $\g[t]$-modules $\{V_1(\zeta_1),...,V_N(\zeta_N)\}$ localized at pairwise distinct points in $\C$, $\{\zeta_1,...,\zeta_N\}$. Denote the chosen cyclic vector of  $V_i(\zeta_i)$ by $v_i$. If $V_i(\zeta_i)$ are finite-dimensional, so is the space $U(\g[t]) v_1\otimes \cdots \otimes v_N$. Moreover, it has a finite filtration by homogeneous degree in $t$.
The Feigin-Loktev fusion product\cite{FL} is the associated
graded space of this filtration. We denote the fusion product
by $\mathcal F^*_{\mathbf V}$. As the grading of $\mathcal F^*_{\mathbf V}$ is $\g$-equivariant, the
graded components are $\g$-modules. The graded multiplicity of the
irreducible $\g$-module $V(\lambda)$ in the fusion product is a certain polynomial generating function for the multiplicities in the graded components. 

The Feigin-Loktev conjecture is that this polynomial is independent of
the localization parameters $\zeta_i$ for sufficiently well-behaved $\g$-modules,
and that in the case that $V_i$ are KR-modules, the graded
multiplicity of $V(\lambda)$ is equal to the
$M$-sum formula \eqref{Msum}. The equality was proven for $\sl_2$-modules in
\cite{FJKLM} and for symmetric power representations of $\sl_n$ in
\cite{Ke2}. 

In this paper, we consider only fusion products of KR-modules. They
are generated by the highest weight vector $v$ localized at $\zeta$
and the relations are those in \eqref{chari}.
Let $\mathbf V = (V_1,...,V_N)$ be a collection of KR-modules of
$\g[t]$ localized at distinct complex numbers $\boldsymbol\zeta =
(\zeta_1,...,\zeta_N)$. 

We parametrize the collection $\mathbf V$ by their highest weights
$\mathbf n=(n_{j}\sal:\ 1\leq \al\leq r, j\geq 0)$, meaning that in
$\mathbf V$ there are exactly $n_{j}\sal$ KR-modules with highest
weight $j \omega_\al$. We call this fusion product $\mathcal F_\bn^*$.

\section{Functional realization of fusion spaces}
We make use of the fact that $\g[t]\subset \widehat{\g}$, therefore
given a $\g[t]$-module $V$, we can consider the integrable modules induced at some fixed integer level $k$.
We choose $k$ to be sufficiently large, so that the tensor products we consider below have
a decomposition determined by the Littelwood-Richardson rule rather than the Verlinde rule.
We choose integrable $\widehat{\g}$-modules since they have the property that they are 
completely reducible. 
Note that smaller
values of $k$ are also of interest, and computing the graded fusion product
at finite $k$ is still an open problem for the most part.

Consider the action of products of (generating functions of) elements of the
affine algebra $\widehat{\g}$ on the tensor product of highest weight
vectors $v_i$ of KR-modules localized at distinct points $\zeta_i$.
We use the generating fuctions $f_{\al}(t) := \sum_{n\in\Z} f_{\al}[n]t^{-n-1}$, where $f_\al$ is
the element of $\n_-$ corresponding to the simple root $\al$.
We define $\mathcal F^*_{\lambda,\bn}=\Hom(\mathcal
F^*_{\bn},V(\lambda))$, where $\bn$ parametrizes the set of KR-modules
in the fusion product. 

The dual space of $\mathcal F^*_{\lambda,\bn}$ 
is the associated graded space of $\mathcal C_{\lambda,\bn}$,
consisting of all correlation functions the form
\begin{equation}\label{cf}
\langle u_\lambda | f_{\al_1}(t_1) \cdots f_{\al_M}(t_M) | v_1\otimes
\cdots \otimes v_N\rangle
\end{equation}
Here,
$M\geq 0$ and $\boldsymbol \al = (\al_1,...,\al_M)\in I_r^M$ where
$I_r=[1,...,r]$.  The action of the currents is the fusion action
of the previous section. The vector $u_\lambda$ is the
lowest weight vector of the module $V$ localized at $\zeta=\infty$,
dual to the highest weight module localized at $0$ with $\g$-highest
weight $\lambda$.

This space has a filtration by the homogeneous degree in $t_j$,
and its associated graded space is the graded multiplicity space of
the $\g$-module
$V(\lambda)$ in the fusion product $\mathcal F^*_{\bn}$.

\subsection{Characterization of functions in $\cC_{\lambda,\bn}$}
We now fix $\lambda$ and $\bn$, and characterize functions in the space
$\cC_{\lambda,\bn}$ according to their symmetry, pole and zero structure\cite{AK}:
\begin{enumerate}
\item {\bf Zero weight condition:} The correlation function \eqref{cf}
  is $\g$-invariant. Therefore it must have total $\g$-weight equal to
  0, which means that
\begin{equation}\label{zeroweight}
0=\ell\sal + \sum_{\beta,j} C_{\al,\beta}m_j^{(\beta)}- \sum_{j}j
n_{j}\sal,\quad 1\leq \al\leq r,
\end{equation}
where $\lambda =\sum_\al \ell\sal \omega_\al$. This fixes $\{m^{(1)},...,m^{(r)}\}$.

For convenience, we rename the variables  to keep
track of the root $\al$ of the generating function in which they
appear. Thus, we have functions in the variables $\{t_1,...,t_M\}=\{t_i^{(\al)}:\
\al\in I_r, 1\leq i \leq m\sal\}$, where $m\sal$ is the number of
generators with root $\al$.
Note that the space $\mathcal C_{\lambda,\bn}$ is the direct sum of spaces of the with fixed
$\bm=(m^{(1)},...,m^{(r)})$.

\item {\bf Pole structure:} Functions in $\mathcal C_{\lambda,\bn}$ have at most a simple pole when
  $t_i^{(\al)}=t_j^{(\beta)}$ if $C_{\al,\beta}<0$. This is due to the
  relations in the algebra, which, in the language of generating functions,
  means that $f_\al(t) f_\beta(u) \sim
  \frac{f_{\al+\beta}(t)}{t-u} + \hbox{non-singular terms}$.
We are therefore led to define the less singular function $g(t)$ for
each $f(t)\in \cC_{\lambda,\bn}$:
\begin{equation}\label{g}
  g(\mathbf t) := \prod_{\al<\beta, C_{\al,\beta}<0}\prod_{i,j}
  (t_i^{(\al)}-t_j^{(\beta)}) f(\mathbf t),\quad f(\mathbf t)\in
  \cC_{\lambda,\bn}.\end{equation}
\item {\bf Symmetry:} 
The function $g(\mathbf t)$ is symmetric under the exchange
$t_i^{(\al)}\leftrightarrow t_j^{(\al)}$. This is due to the fact
that $[f_\al(t_1),f_\al(t_2)]=0$.
\item {\bf Serre condition:} Let $\al,\beta$ be simple roots such that
  $C_{\al,\beta}<0$ and define
  $m_{\al,\beta}=1-C_{\al,\beta}$. Then there is a Serre relation in $\g$
  (hence a corresponding relation in $\widehat{\g}$) of the form
${\rm ad}(f_\al)^{m_{\al,\beta}} f_\beta=0$. In generating function
language, 
$$
f_\al(t_1^{(\al)})\cdots f_\al(t_{m_{\al,\beta}}^{(\al)})t_\beta(t_1^{(\beta)})
$$
has no singularity when all the variables are set equal to each other.
This implies that
the function $g(\bt)$ of \eqref{g} has the following vanishing property:
$$g(\mathbf t)\Big|_{t_1^{(\al)}=\cdots = 
  t_{m_{\al,\beta}}^{(\al)}=t_{1}^{(\beta)}}=0.$$ 
This cancels out the pole which would otherwise appear in the function
$f(\bt)$.

\item {\bf Degree restriction:} As $u_\lambda$ is a lowest weight
  vector of the module localized at infinity, positive currents
  $f_\al[n]_\infty$ with $n\geq 0$ act on it trivially. The action is
  given by taking the contour integral at infinity (see Equation
  \eqref{inftyaction}), or equivalently, a residue taken at 0. That
  is, there should be no residue when integrating $t^{n-2} f(t^{-1})$, with
  $n\geq 0$, at $t=0$.
This
  gives a degree restriction on the function $f(\mathbf t)\in
  \cC_{\lambda,\bn}$ for each of the variables:
$$
{\rm deg}_{t_i\sal} f(\mathbf t)\leq -2.
$$
\item {\bf Poles at $\zeta_i$:} The relation \eqref{highest} implies
  that $f(\mathbf t)\in \cC_{\lambda,\bn}$ may have a simple pole at
  $t_i\sal=\zeta_j$ only if the highest weight of the module localized at
  $\zeta_j$ is a multiple of $\omega_\al$, in accordance with the Equation
  \eqref{localaction}. Otherwise, $f[n]_{\zeta_i} v(\zeta_i)=0$ if
  $n\geq 0.$ Moreover, we have
\item {\bf Integrability condition:} We assume each module $V_k$ has highest
  weight $\ell_k \omega_{\al_k}$. The relation \eqref{chari} requires
  that $g_2(\bt) = \left(\prod_{\al,j,k} (t_j\sal-\zeta_k)^{\delta_{\al,\al_k}}\right)
g(\bt)$ has the following vanishing property:
$$
g_2(\bt)\Big|_{t_1\sal = \cdots = t_{\ell_k+1}\sal=\zeta_k}=0.
$$

\end{enumerate}

These conditions characterize the space $\cC_{\lambda,\bn}$ completely.
The only difficulty is to compute its Hilbert polynomial.
This is done by introducing another filtration on
the space of functions. The idea for such a filtration was first
introduced by Feigin and Stoyanovsky \cite{FeiginSto}.

\subsection{Filtration of the space of functions}
Let $\blambda = (\lambda^{(1)},...,\lambda^{(r)})$ where $\lambda\sal$
is a partition of $m\sal$. Let $m\sal_a$ denote the number of parts of
$\lambda\sal$ equal to $a$. Thus, $\sum_a a m\sal_a=m\sal$. Fix a
standard tableau for each partition (the result is independent of the
choice of tableaux, and when we discuss a partition below we always refer to
the fixed tableau) and define the evaluation map $\ev_\lambda :
\cC_{\lambda,\bn}[m^{(1)},...,m^{(r)}] \to \mathcal H[\blambda]$,
where $\mathcal H[\blambda]$ is the space of functions in several
variables: one variable for each row of each partition in
$\blambda$. 

The evaluation map is defined as follows. If the letter $i$ appears in
the $j$th row of length $a$ in $\lambda\sal$, then
$\ev_\blambda(t_i\sal) = u\sal_{a,j}$. This is extended by linearity to
$\cC_{\lambda,\bn}$. 

We order multipartitions lexicographically, and define 
$$
\Gamma_\blambda= \underset{\bmu>\blambda}{\cap} \ker\ev_\bmu.
$$
This gives a finite filtration of $\cC_{\lambda,\bn}$, with
$\Gamma_\bmu \subset \Gamma_\blambda$ if $\bmu<\blambda$. We consider the
image of the graded components
$
\Gamma_\blambda/(\Gamma_\blambda\cap \ker \ev_\blambda)
$
under the evaluation map $\ev_\blambda$.

Again, this is a space of functions, isomorphic to a subspace of
$\mathcal H[\blambda]$. Let us denote its image by
$\widetilde\cH[\blambda]$. Its characterization is as follows.

\begin{enumerate}
\item {\bf Symmetry:} Functions in $\Gamma_\blambda$ are symmetric in
  the variables $\{t_1\sal,...,t_{m\sal}\sal\}$ for each $\al$. The
  full symmetry is lost under the evaluation map, but the functions
  are still symmetric with respect to the variables labeled by rows of
  the same length in 
  $\lambda\sal$. That is, they are symmetric with respect to the
  exchange of
  variables  $\{u_{a,1}\sal,...,u_{a,m_a\sal}\sal\}$ for each $a,\al$.

\item Functions in $\Gamma_\blambda$ are in the kernel of any evaluation
$ev_\bmu$ with $\bmu>\blambda$, which means that functions in the
image vanish whenever we set the variables corresponding to different
rows of the same partition equal to each other. In fact, one can prove
that 
\begin{lemma}
Functions in $\widetilde\cH[\blambda]$ have a factor
$(u\sal_{a,j}-u\sal_{a,k})^{2\min(a,b)}$ for all $j<k$. 
\end{lemma}

\item {\bf Pole structure and Serre condition:} The pole at
  $t_i\sal=t_j^{(\beta)}$ when $C_{\al,\beta}<0$, 
  together with the vanishing condition of $g(\bt)$ which follows from
  from the Serre relation, implies that functions in
  $\widetilde\cH[\blambda]$ have a pole of order at most $\min(|C_{\al,\beta}| b,
  |C_{\al,\beta}|a)$ whenever $u_{a,i}\sal=u_{b,j}\sbe$ (inherited
  from conditions (1) and (3) of the previous subsection).
\item {\bf Poles at $\zeta_j$:} The pole at $t_i\sal = \zeta_j$ in
  case $V_j$ has highest weight proportional to $\omega_\al$, together
  with the integrability of that module, translate to the following
  statement for $f\in \widetilde\cH[\blambda]$: There is a pole of
  order at most $\min(\ell,a)$ at $u_{a,i}\sal=\zeta_j$ if $V_j$ has
  a highest weight equal to $\ell_j \omega_\al$. 
We define
  $\delta(j,\al)=1$ if the highest weight of $V_j$ is a multiple of
  $\omega_\al$, and $\delta(j,\al)=0$ otherwise.

\item {\bf Degree restriction:} Functions in $\widetilde\cH[\blambda]$
  have a degree in $u_{a,j}\sal$ which is bounded from above by $-2 a$.
\end{enumerate}

We do not know that these are all the conditions on functions in
$\widetilde\cH[\blambda]$: The map $ev_\blambda$ is injective by definition but not necessarily surjective.
However, we can compute the Hilbert polynomial
of the space $\mathcal F$ defined by the conditions above, which gives
an upper bound on the Hilbert polynomial of $\widetilde\cH[\blambda]$.

To summarize, we know that $f(\bu)\in \mathcal F$ has the form
$$
\frac{\displaystyle \prod_{(a,i)\neq(b,j)}
  (u_{a,i}\sal-u_{b,j}\sal)^{\min(a,b)}\times f^0(\bu)}
{\displaystyle\prod_{a,i,j}(u_{a,i}\sal-\zeta_j)^{\delta(j,\al) \min(\ell_j,a)}
\prod_{a,i,b,j}  \prod_{\al<\beta:C_{\al,\beta}<0}
(u_{a,i}\sal-u_{b,j}\sbe)^{\min(|C_{\al,\beta}|b,|C_{\beta,\al}|a)}
},
$$
where $f^0(\bu)$ is a polynomial in $\bu$, symmetric under the exchange
$u\sal_{a,i}\leftrightarrow u\sal_{a,j}$, of degree such that 
$$
\deg_{u_{a,j}\sal} f(\bu) \leq -2a.
$$
To compute the Hilbert polynomial we set all $\zeta_j=0$ so that the
function above is homogeneous in $\bu$. That is, we compute the
Hilbert polynomial of the associated graded space. It is very important to note that 
the values of $\zeta_j$ do not affect the value of the Hilbert polynomial, that is,
there is no change in the $q$-dimensions of the space when we take the associated
graded space.

The degree in $u_{a,j}\sal$ of the
prefactor of $f^0(\bu)$ is $-2a-P_a\sal$, where $P_a\sal$ is defined
in Equation \eqref{vacancy}. Moreover, the overall homogeneous degree
of the prefactor
is $Q(\bm,\bn)$ as defined in equation \eqref{energy}. The Hilbert
polynomial of the space of symmetric functions in $m$ variables of
degree less than or equal to $p$ is the $q$-binomial coefficient,
$$
\left[ \begin{array}{c} m+p \\ m \end{array} \right]_q =
\frac{\displaystyle \prod_{i=1}^{m+p}(1-q^i)} {\displaystyle
  \prod_{i=1}^{m}(1-q^i)\prod_{i=1}^{p}(1-q^i)} .
$$
Therefore, the Hilbert polynomial of $\mathcal F$ is
$$
q^{Q(\bm,\bn)} \prod_{\al,j}\left[ \begin{array}{c} m\sal_j+p\sal_j \\
    m\sal_j \end{array} \right]_q,
$$
which is the upper bound (at each degree in $q$) of the Hilbert
polynomial of $\widetilde{\cH}[\blambda]$, since it is a polynomial with
positive coefficients.

Summing over the graded components, there follows the main Theorem:
\begin{theorem}\cite{AK}
The Hilbert polynomial of $\cC_{\lambda,\bn}$, which is the
Feigin-Loktev fusion product of KR-modules, is bounded from
above by $M_{\lambda,\bn}(q)$ defined
  in Equation \eqref{Msum}.
\end{theorem}
\section{Proof of the $M=N$ conjecture}
In this section, we explain the proof \cite{DFK} of the identity
\eqref{HKOTY}. For ease of readability, we explain the technique explicitly
for the Lie algebra $\sl_2$, and then state the key ingredients
necessary in the generalization to arbitrary Lie algebras. The only
difficulty in this generalization is the rapid proliferation of indices.

\subsection{The case of $\sl_2$}
As explained in the introduction, one need only prove the $M=N$
identity only for the case $q=1$ for the pentagon of identities to
hold, due to the positivity of the $M$-sum. In the case of $\sl_2$, we drop the root superscript $(\al)$ in the vacancy numbers $P_i\sal$ and so forth.

Fix $\bn=(n_1,...,n_k)\in \Z_+^k$ and an $\sl_2$-highest weight $\ell
\omega_1$ with $\ell\in \Z_+$. Consider the following generating function:
\begin{equation}\label{zfun}
Z_{\ell,\bn}^{(k)}(x_0,x_1) = \sum_{\bm\in \N^k} x_1^{-q_0} x_0^{q_1} \prod_{i=1}^k {m_i + q_i \choose m_i}
\end{equation}
Here, we have defined  
$$q_i = \ell + \sum_{j=i+1}^k (j-i) (2 m_j-n_j),\quad i\geq 0.$$ 
In particular, notice that when $q_0=0$, $q_i=P_i$ for all $i>0$.

The binomial coefficient is defined as usual
$$
{m+p\choose m} = \frac{(m+p)!}{m!p!}.
$$
This is well-defined for both negative and positive values of $p$, and
when $p<0$ it has an overall sign $-1^m$.

This  $N$-sum can be obtained from this generating function as
follows. First, here and below, we note that in the $N$ and $M$-sums,
$m_j=0$ if $j>k$ in \eqref{Msum}. However all the identities we prove
are valid under this restriction; since only a finite number of the
$m_j$ make a non-trivial contribution to the summation \eqref{Msum},
one can take $k\to\infty$ at the end of the day with no loss of generality.

Second, in both the $N$ and $M$ sums, there is a ``weight restriction"
restriction on the $\bm$-summation. This is equivalent to setting
$q=0$, or alternatively, considering only the constant term in $x_1$
in the generating function.  expression. We do not restrict the sum to
$P_i\geq 0$ yet, but in the $M$ and $N$ sums, the variable $x_0$ must
be set to 1.

\begin{lemma}
There is a recursion relation,
$$
Z_{\ell,(n_1,..,n_k)}^{(k)}(x_0,x_1) = \frac{x_1^{n_1+2}}{x_0 x_2} Z_{\ell;(n_2,...,n_k)}^{(k-1)}(x_1,x_2),
$$
where $x_i$ are solutions of the $A_1$ $Q$-system or cluster algebra
mutation\cite{Ke} with arbitrary boundary conditions:
$$
x_{i+1} x_{i-1} = x_i^2-1, \quad i\in \Z.
$$
\end{lemma}
\begin{proof}
  The variable $m_1$ is not part of the expression for $q_1$ so we can
  perform the summation over $m_1$, using the identity
$$
\sum_{m_1\geq 0} x_1^{-2m_1} {m_1+q_1\choose m_1} = \left(\frac{x_1^2}{x_1^2-1}\right)^{q_1+1}= \left(\frac{x_1^2}{x_{0}x_{2}}\right)^{q_1+1},
$$
where we have used the $Q$-system in the second equality.

We separate out the dependence on $m_1$ in the summand, and note that
$$2q_{i}-q_{i-1}+2m_i-n_i=q_{i+1}.$$ 
Moreover, if we denote by $q_i^{(j)}$ the function $q_i$ with
arguments being of the last $j-i$ variables in the list
$(m_1,...,m_k)$ (so that $q_i^{(k)}=q_i$), then $q_i^{(k-1)}=q_{i+1}^{(k)}$, or
$q_{i-1}^{(k-1)}=q_i^{(k)}$.

We have
\begin{eqnarray*}
Z_{\ell,(n_1,...,n_k)}^{(k)} (x_0,x_1)&=& \sum_{m_1,...,m_k} x_1^{-q_0} x_0^{q_1} \prod_{j=1}^k {m_j+q_j\choose m_j} \\
& & \hskip-1in
=\sum_{m_2,...,m_k} x_0^{q_1}\prod_{j\geq 2} {m_j+q_j\choose m_j} 
x_1^{n_1+q_2-2q_1}
\sum_{m_1}  x_1^{-2 m_1}{m_1+q_1\choose m_1}   \\
&=& \sum_{m_2,...,m_k} x_0^{-1} x_1^{n_1+q_2+2} x_2^{-1-q_1} \prod_{j=2}^k {m_j + q_{j-1}^{(k-1)}\choose m_j}\\
&=&\frac{x_1^{n_1+2}}{x_0 x_2} \sum_{m_2,...,m_k} x_2^{-q_0^{(k-1)}}x_1^{q_1^{(k-1)}} \prod_{j=1}^{k-1} {m_{j+1} +q_j^{(k-1)}\choose m_{j+1}}\\
&=& \frac{x_1^{n_1+2}}{x_0 x_2} Z_{\ell, (n_2,...,n_k)}^{(k-1)}(x_1,x_2).
\end{eqnarray*}
(Here, the superscript $(k-1)$ on $q_0,q_1$ means we take these variables as defined for the $k-1$ variables with indices $2,...,k$.)
\end{proof}

Using the Lemma, by induction, we see that the generating function factorizes: 
\begin{eqnarray}\label{factora}
Z_{\ell,\bn}^{(k)}(x_0,x_1) &=& \frac{x_1 x_k^{\ell+1}}{x_0 x_{k+1}^{\ell+1}} \prod_{i=1}^k x_{i}^{n_i}.
\end{eqnarray}
In particular,
\begin{eqnarray}\label{factor}
Z_{\ell,\bn}^{(k)}(x_0,x_1) &=& Z_{0,(n_1,...,n_p)}^{(p-1)}(x_0,x_1) Z_{\ell,(n_{p+1},...,n_k)}^{(k-p+1)}(x_{p-1},x_{p}).
\end{eqnarray}
We are interested in the constant term in $x_1$ in $Z_{\ell,\bn}^{(k)}(x_0,x_1)$. 
We use the factorization Lemma for the first factor, and the definition via summation for the second factor:
\begin{eqnarray}\label{factorb}
Z_{\ell,\bn}^{(k)}(x_0,x_1) &=& \frac{x_1 x_{p-1}}{x_0 x_{p}} \prod_{j=1}^{p-1} x_j^{n_j}
\sum_{m_{p},...,m_k} x_{p}^{-q_{p-1}} x_{p-1}^{q_p} \prod_{j=p}^{k} {m_{j} + q_j\choose m_{j}}.
\end{eqnarray}

Suppose we restrict the summation in the second factor to $q_p\geq 0$ only. Moreover, we are interested in the generating function when $x_0=1$. In this case, all $x_i$ are polynomials in $x_1$ (Chebyshev polynomials of the second type). Terms in the summation in which $q_{p-1}<0$ are therefore products of polynomials in $x_1$ since there are no factors of $x_i$ in the denominator in this case. Moreover, there is an overall factor of $x_1$, so that there is no constant term in $x_1$ in this case. Thus, 
\begin{lemma}
If the summation over $(m_p,...,m_k)$ in \eqref{factorb}, is
restricted to $q_p\geq 0$, then only terms with $q_{p-1}\geq 0$
contribute to the constant term in $x_1$ when $x_0=1$.
\end{lemma}
We use an induction argument, where the base step is clear
($q_k=\ell$), to conclude that the only terms which contribute to the
constant term in $x_1$ are terms from the restricted summation,
$q_i\geq 0$ ($i>0$). When $q_0=0$, this is the $N=M$ identity, since
$q_i=P_i$ in that case.

\subsection{The simply-laced case}
This case is a straightforward generalization of the $\sl_2$ case\footnote{Below, we have two sets of indices for $\bx$, $\bn$ etc. When we write
$\bx_0$ we mean the collection of $r$ elements
$(x_{1,0},...,x_{r,0})$, and so forth.}.

We now define the generating function
$$
Z_{\lambda,\bn}^{(k)}(\bx_0,\bx_1)= \sum_{\bm} \bx_0^{-\bq_1} \bx_1^{\bq_0} \prod_{\al,j} {m_{j}\sal+q_{j}\sal\choose m_{j}\sal},
$$
(as is the norm, when $\bx$ and $\bq$ represent vectors indexed by the
same set, we write $\bx^\bq$ for the product over the components.)
Here, $\lambda = \sum_{\al=1}^r \ell\sal \omega_\al$,
$\bn=(n_{j}\sal)_{\al\in I_r,j\in I_k}$, the summation is over $\bm =
\{m_{j}\sal,\ \al\in I_r, j\in I_k\}$ non-negative integers, and we
define
$$
q_{i}\sal = \ell\sal + \sum_{j=i+1}^k \sum_{\beta\in I_r} (j-i) (C_{\al,\beta} m_{j}\sbe-\delta_{\al,\beta} n_{j}\sbe).
$$
When $q_{\al,0}=0$ for all $\al$, this corresponds to the ``weight
restriction" \eqref{weight} in the $M$ and $N$-sums, and in that case,
$q_{i}\sal=P_{i}\sal$ if $i>0$. We have now $2r$ variables $\bx_0 =
(x_{1,0},...,x_{r,0})$ and $\bx_1=(x_{1,1},...,x_{r,1})$. The
generating function is related to the $M$ or $N$-sums when we evaluate
the sum
at $x_{\al,0}=1$ and consider the constant term in $\bx_{1}$.

Following the steps outlined for $\sl_2$ we derive a recursion relation for the generating function:
$$
Z_{\lambda,\bn}^{(k)}(\bx_0,\bx_1) =\frac{\bx_{1}^{2+\bn_{1}}}{\bx_{0}\bx_{2}}Z_{\lambda,\bn^{(k-1)}}^{(k-1)}(\bx_1,\bx_2)
$$
where $\bn^{(k-1)}$ is $\bn$ with $\bn_{1}=\boldsymbol 0$. Here, we
have defined $x_{\al,i}$ to be the solutions of the following system:
$$
x_{\al,i+1} x_{\al,i-1} = x_{\al,i}^2 - \prod_{C_{\al,\beta}=-1} x_{\beta,i}.
$$
This is called the $Q$-system for the simply-laced Lie algebra $\g$,
provided we set the initial conditions $x_{\al,0}=1$. Otherwise it is
a cluster algebra mutation\cite{Ke}, and therefore, under these special initial
conditions, all its solutions are polynomials in the variables
$x_{\beta,1}$\cite{DFK2}.

We again repeat the arguments of the previous section to factorize the
generating function:
$$
Z_{\lambda,\bn}^{(k)}(\bx_0,\bx_1) = \prod_\al \frac{x_{\al,1} x_{\al,k}^{\ell_\al+1}}{x_{\al,0}x_{\al,k+1}^{\ell_\al+1}}
\prod_{j=1}^k x_{\al,j}^{n\sal_{j}},
$$
from which we deduce that
$$
Z_{\lambda,\bn}^{(k)}(\bx_0,\bx_1) = \frac{\bx_{1} \bx_{{p-1}}}{\bx_0\bx_{p}}
\prod_{j=1}^{p-1} \bx_{j}^{\bn_{j}}
\sum_{\bm^{(p)}} \bx_p^{-\bq_{p-1}} \bx_{p-1}^{\bq_p} \prod_{j= p}^k {\bm_{j}+\bq_{j}\choose \bq_{j}}.
$$
Here, $\bm^{(p)}$ are the last $k-p+1$ components of the list
$(\bm_1,...,\bm_k)$. A binomial coefficient with vector entries is
notation for the product of binomial coefficients over the components.
 
Suppose we restrict the summation to $\bm^{(p)}$ such that
$q\sal_{p}\geq 0$ for some $\al$, and such that $q\sal_{p-1}<0$ for
the same $\al$. We look for a contribution to the constant term in
$x_{\al,1}$. All $\bx_i$ are polynomials in $\bx_1$ after evaluation
at $x_{\al,0}=1$ for all $\al$. Terms with $q\sal_{p}<0$ do not have a
factor $x_{\al,p}$ in the denominator, and are therefore polynomials
in $x_{\al,i}$ for several $i$ and fixed $\al$. One can show that
$\prod_{\beta\neq \al} x_{\beta,p}^{-1}$ has no negative powers of
$x_{\al,1}$ (see \cite{DFK}, Lemma 4.8).  Therefore we have a
polynomial in $x_{\al,1}$, with an overall power of $x_{\al,1}$, hence
there is no constant term in $x_{\al,1}$.  We repeat this argument for
each $\al$ and inductively for each $p$ starting from $p=k$, until we
get
\begin{lemma}
There is no contribution to the constant term in $\bx_1$ in the
summation from terms with $q\sal_{j}<0$ for any $p,j$, hence from
terms with $P_{j}\sal<0$ when we consider the terms with $q_{0}\sal=0$.
\end{lemma}
This implies that $M=N$ for the simply-laced Lie algebras.

\subsection{The non-simply laced case}
This case is less elegantly derived, as it requires the introduction of
even more variables in the generating function, and each case must be
treated separately. Nevertheless, the argument goes through in the
same (more involved) manner. In the process we must define the set of
variables which satisfy the following system of equations:
$$
x_{\al,i+1}x_{\al,i-1} = x_{\al,i} - \prod_{\beta:C_{\al,\beta}<0}
\prod_{j=0}^{-C_{\al,\beta}-1}
x_{\beta,\lfloor|(C_{\beta,\alpha|}i+j)/|C_{\al,\beta}|\rfloor}. 
$$
If $x_{\al,0}=1$ for all $\al$, then the equation for $i>0$ is known
as the $Q$-system (for the simple Lie algebra with Cartan matrix $C$),
and it is known to be satisfied by the characters $x_{\al,i}$ of the
KR-modules $\KR_{\al,i}$ if $x_{\al,1}$ is the character of the
fundamental module.

We find in the $k\to\infty$ limit that
\begin{theorem}\cite{DFK}
For any simple Lie algebra and $\lambda$ a dominant weight, $\bn$ a
vector in $\Z_+^{r\times k}$, $M_{\lambda,\bn}=N_{\lambda,\bn}$.
\end{theorem}

\section{Summary}
Prior to the work described in the previous section, it was known that for any simple Lie algebra, the multiplicity of the $U_q(\g)$-module with $\g$ highest weight $\lambda$ in the tensor product of Kirillov-Reshetikhin modules is the $N$-sum formula. This followed theorems of Hatayama et al\cite{HKOTY} and Nakajima's theorem about the $q$-characters for $T$-systems corresponding to simply-laced Lie algebras\cite{Nakajima}, as well as the extension by Hernandez for other algebras\cite{Hernandez}.

We now have all the equalities in the pentagon of identities. That is, since we have proven that $M=N$, we have proven also the following:
\begin{corollary}
The multiplicity of the $\g$-module $V(\lambda)$ in the tensor product of Chari's KR modules of $\g[t]$ is equal to the multiplicity of the $U_q(\g)$ module with $\g$-highest weight $\lambda$ in the corresponding tensor product of $U_q(\widehat{\g})$ of Kirillov-Reshetikhin type.
\end{corollary}
\begin{corollary}
The Hilbert polynomial of the graded multiplicity space of $V(\lambda)$ in the Feigin-Loktev fusion product is the fermionic $M$-sum (generalized Kostka polynomials in the case of $A_n$). This is the Feigin-Loktev conjecture.
\end{corollary}
\begin{corollary}
The Bethe integer sets (parametrizing Bethe vectors) in the generalized Heisenberg model as solved by Kirillov and Reshetikhin are in bijection with vectors in the Hilbert space of the model, and therefore the completeness conjecture holds.
\end{corollary}
We should remark that although it is well known that not all Bethe states come from the so-called ``string hypothesis" in these models, nevertheless this gives a good counting of the states.

The proof described in the previous section shows that the vanishing
of the ``non-positive'' components of the $N$-sum formula is due to
the fact that the solutions of the $Q$-system with the KR-boundary
condition are polynomials in the initial data $x_{\al,1}$. This fact
is clear, once one refers to the theorem that the solutions
$x_{\al,n}$ with $n>0$ are characters of KR-modules, which are in the
Grothendieck group generated by $\{x_{1,1},...,x_{r,1}\}$ (the
characters of the $r$ fundamental representations). However these
facts are not immediately obvious without resorting to the proven
theorems on the subject.  The cluster algebra formulation of the
$Q$-system gives an entirely
combinatorial interpretation for this fact \cite{Ke,DFK2}.

The polynomiality property is quite general for a much larger class of
cluster algebras, under even more general boundary conditions, which
give a certain vanishing of the numerators \cite{DFK2}. For example,
the same property holds for the $T$-systems. 

\section{Acknowlegments} 
I am deeply indebted to T. Miwa for so generously sharing his
knowledge and experience over the span of many years. I also thank my collaborators in completing the series of proofs, Eddy Ardonne and Philippe Di Francesco. This work is
supported by NSF grant DMS-0802511.

\bibliographystyle{ws-procs9x6}
\bibliography{refs}

\def\cprime{$'$} \def\cprime{$'$}
\begin{thebibliography}{10}

\bibitem{AKS}
E.~Ardonne, R.~Kedem and M.~Stone, {\em J. Phys. A} {\bf 38}, 9183 (2005).

\bibitem{AK}
E.~Ardonne and R.~Kedem, {\em J. Algebra} {\bf 308}, 270 (2007).

\bibitem{DFK}
P.~D. Francesco and R.~Kedem, {\em Int. Math. Res. Not. IMRN} , Art. ID rnn006,
  57 (2008).

\bibitem{FKLMM}
B.~Feigin, R.~Kedem, S.~Loktev, T.~Miwa and E.~Mukhin, {\em Compositio Math.}
  {\bf 134}, 193 (2002).

\bibitem{FJKLM}
B.~Feigin, M.~Jimbo, R.~Kedem, S.~Loktev and T.~Miwa, {\em Duke Math. J.} {\bf
  125}, 549 (2004).

\bibitem{FL}
B.~Feigin and S.~Loktev, On generalized {K}ostka polynomials and the quantum
  {V}erlinde rule, in {\em Differential topology, infinite-dimensional {L}ie
  algebras, and applications\/}, , Amer. Math. Soc. Transl. Ser. 2 Vol.~194
  (Amer. Math. Soc., Providence, RI, 1999) pp. 61--79.

\bibitem{KR}
A.~N. Kirillov and N.~Y. Reshetikhin, {\em Zap. Nauchn. Sem. Leningrad. Otdel.
  Mat. Inst. Steklov. (LOMI)} {\bf 160}, 211 (1987).

\bibitem{HKOTY}
G.~Hatayama, A.~Kuniba, M.~Okado, T.~Takagi and Y.~Yamada, Remarks on fermionic
  formula, in {\em Recent developments in quantum affine algebras and related
  topics ({R}aleigh, {NC}, 1998)\/}, , Contemp. Math. Vol.~248 (Amer. Math.
  Soc., Providence, RI, 1999) pp. 243--291.

\bibitem{Nakajima}
H.~Nakajima, {\em Represent. Theory} {\bf 7}, 259 (2003).

\bibitem{Hernandez}
D.~Hernandez, {\em J. Reine Angew. Math.} {\bf 596}, 63 (2006).

\bibitem{Chari}
V.~Chari, {\em Internat. Math. Res. Notices} , 629 (2001).

\bibitem{KKR}
S.~V. Kerov, A.~N. Kirillov and N.~Y. Reshetikhin, {\em Zap. Nauchn. Sem.
  Leningrad. Otdel. Mat. Inst. Steklov. (LOMI)} {\bf 155}, 50 (1986).

\bibitem{KSS02}
A.~N. Kirillov, A.~Schilling and M.~Shimozono, {\em Selecta Math. (N.S.)} {\bf
  8}, 67 (2002).

\bibitem{Sch05}
A.~Schilling, {\em J. Algebra} {\bf 285}, 292 (2005).

\bibitem{OSS03}
M.~Okado, A.~Schilling and M.~Shimozono, A crystal to rigged configuration
  bijection for nonexceptional affine algebras, in {\em Algebraic combinatorics
  and quantum groups\/},  (World Sci. Publ., River Edge, NJ, 2003) pp. 85--124.

\bibitem{SS06}
A.~Schilling and M.~Shimozono, {\em J. Algebra} {\bf 295}, 562 (2006).

\bibitem{Her}
D.~Hernandez, {\em J. Reine Angew. Math.} {\bf 596}, 63 (2006).

\bibitem{KNS}
A.~Kuniba, T.~Nakanishi and J.~Suzuki, {\em Internat. J. Modern Phys. A} {\bf
  9}, 5215 (1994).

\bibitem{KirillovShi}
A.~N. Kirillov and M.~Shimozono, {\em J. Algebraic Combin.} {\bf 15}, 27
  (2002).

\bibitem{ChariMoura}
V.~Chari and A.~Moura, {\em Comm. Math. Phys.} {\bf 266}, 431 (2006).

\bibitem{Ke2}
R.~Kedem, {\em Int. Math. Res. Not.} , 1273 (2004).

\bibitem{FeiginSto}
A.~V. Stoyanovski{\u\i} and B.~L. Fe{\u\i}gin, {\em Funktsional. Anal. i
  Prilozhen.} {\bf 28}, 68 (1994).

\bibitem{Ke}
R.~Kedem, {\em J. Phys. A} {\bf 41}, 194011, 14 (2008).

\bibitem{DFK2}
P.~Di~Francesco and R.~Kedem, {\em Lett. Math. Phys.} {\bf 89}, 183 (2009).

\end{thebibliography}

\end{document}